\documentclass[graybox]{svmult}

\usepackage{mathptmx}       % selects Times Roman as basic font
\usepackage{helvet}         % selects Helvetica as sans-serif font
\usepackage{courier}        % selects Courier as typewriter font
\usepackage{type1cm}        % activate if the above 3 fonts are
\usepackage{makeidx}         % allows index generation
\usepackage{graphicx}        % standard LaTeX graphics tool
\usepackage{multicol}        % used for the two-column index
\usepackage[bottom]{footmisc}% places footnotes at page bottom

\usepackage{amssymb}
\usepackage{amsmath}

\newtheorem{Remark}{Remark}
\newtheorem{Corollary}{Corollary}
\newtheorem{Definition}{Definition}
\newtheorem{Problem}{Problem}
\newtheorem{Theorem}{Theorem}
\newtheorem{Lemma}{Lemma}

\newtheorem{Assumption}{Assumption}
\newtheorem{Property}{Property}

\newcommand{\bpropty}{\begin{Property}}
\newcommand{\epropty}{\end{Property}}

\newcommand{\bdefinition}{\begin{Definition}}
\newcommand{\edefinition}{\end{Definition}}

\newcommand{\bremark}{\begin{Remark}}
\newcommand{\eremark}{\end{Remark}}

\newcommand{\basm}{\begin{Assumption}}
\newcommand{\easm}{\end{Assumption}}

\newcommand{\blemma}{\begin{Lemma}}
\newcommand{\elemma}{\end{Lemma}}

\newcommand{\btheorem}{\begin{Theorem}}
\newcommand{\etheorem}{\end{Theorem}}

\newcommand{\EQQ}{\begin{eqnarray*}}
\newcommand{\ENN}{\end{eqnarray*}}
\newcommand{\EQ}{\begin{eqnarray}}
\newcommand{\EN}{\end{eqnarray}}

\newcommand{\nnum}{\nonumber}

\newcommand{\col}{\mbox{col}}
\newcommand{\R}{{\mathbb R}}

\makeindex             % used for the subject index
\begin{document}

\title*{Certainty Equivalence, Separation Principle, and Cooperative Output Regulation of Multi-Agent Systems by Distributed Observer Approach}
\author{Jie~Huang}
\institute{
%Name of First Author
\at Jie~Huang, Department of Mechanical and Automation
Engineering, The Chinese University of Hong Kong, Shatin, N.T., Hong
Kong, \email{jhuang@mae.cuhk.edu.hk.}
This book chapter was submitted to the book "Recent Advances and Future Directions on Adaptation and Control", Elsevier, in early May 2015
%\and Name of Second Author \at Name, Address of Institute \email{name@email.address}
}
\maketitle

\abstract*{Each chapter should be preceded by an abstract (10--15 lines long) that summarizes the content. The abstract will appear \textit{online} at \url{www.SpringerLink.com} and be available with unrestricted access. This allows unregistered users to read the abstract as a teaser for the complete chapter. As a general rule the abstracts will not appear in the printed version of your book unless it is the style of your particular book or that of the series to which your book belongs.
Please use the 'starred' version of the new Springer \texttt{abstract} command for typesetting the text of the online abstracts (cf. source file of this chapter template \texttt{abstract}) and include them with the source files of your manuscript. Use the plain \texttt{abstract} command if the abstract is also to appear in the printed version of the book.}

\abstract{The cooperative output regulation problem of  linear multi-agent systems was formulated and studied by the distributed observer approach in \cite{sh1, sh2}.
Since then, several variants and extensions have been proposed, and the technique of the distributed observer has also been applied to such problems as formation, rendezvous, flocking, etc.
In this chapter, we will first present a more general formulation of  the cooperative output regulation problem for linear multi-agent systems that
includes some existing versions of the cooperative output regulation problem as special cases.
Then, we will describe  a more general distributed observer. Finally, we will simplify the proof of the main results by more explicitly utilizing the separation principle and
the certainty equivalence principle.}

\section{Introduction} \label{sec1.1}

 The cooperative output regulation problem by distributed observer approach was first studied for linear multi-agent systems subject to  static communication topology in \cite{sh1}, and then for  linear multi-agent systems subject to dynamic communication topology in
\cite{sh2}.
The problem is interesting because its formulation includes the leader-following consensus, synchronization or formation as special cases.
In contrast with the output regulation problem of a single linear system \cite{Da,Fr,FW}, the information of the exogenous signal may not be available for every subsystem due to the communication constraints. Thus, information sharing, or, what is the same, cooperation among different subsystems is essential in the design of the control law.
We call a control law that satisfies the communication constraints as a distributed control law.
The core of the approach in \cite{sh1, sh2} is the employment of the so-called distributed observer, which provides the estimation of the leader's signal to each follower so that a distributed controller can be synthesized based on a purely decentralized controller and the distributed observer. Such an approach to designing  a distributed controller is known as  the certainty equivalence principle.

Since the publication of \cite{sh1, sh2}, several variants and extensions of \cite{sh1, sh2} have been proposed \cite{lfg, ldc, mydj,sra, SYSCL12}.
The objectives of this chapter are three folds. First, we will present a more general formulation of the cooperative output regulation problem for linear multi-agent systems that
includes some existing versions of the cooperative output regulation problem as special cases.
Second, we will describe a more general distributed observer. Third, we will simplify the proof of the main results by more explicitly utilizing the separation principle and
the certainty equivalence principle.

The cooperative output regulation problem by the distributed observer approach can also be generalized to some nonlinear systems such as multiple Euler-Lagrange systems
\cite{cai-huang-2014-ijgs} and multiple rigid-body systems \cite{cai-huang-2014-auto}. Moreover, the distributed observer approach
can also be applied to such problem as the leader-following flocking / rendezvous with connectivity preservation
\cite{DHautomatica, DHCDC14}.

It should be noted that the cooperative output regulation problem of multi-agent systems can also be handled by the distributed internal model approach  \cite{shh,Jiang1}. This approach has an additional advantage that it can tolerate
perturbations of the plant parameters, and it does not need to solve the regulator equations. A combined distributed internal model and distributed observer approach is proposed in \cite{ldc}.
Nevertheless, in this chapter, we will only focus on the distributed observer approach.

{\bf Notation.} $\otimes$
denotes the Kronecker product of matrices.
%Vector $\mathbbold{1}_N$
%denotes a $N$ dimensional column vector with all elements $1$.
%$||x||$ denotes the Euclidean norm of vector $x$. For a square matrix $A$,  $||A||$ denotes
%the induced norm of $A$ by the Euclidean norm, and
$\sigma(A)$ denotes the spectrum of a square matrix $A$.
For matrices $x_i\in \R^{n_i \times p}$, $i=1,\dots,m$,
col$(x_1,\dots,x_m)=[x_1^T,\dots,x_m^T]^T$.
We use
$\sigma(t)$ to denote
 a piecewise constant switching signal
$\sigma :~[0,+\infty)$ $\rightarrow \mathcal {P}=\{1,2,\dots,\rho\}$ for
some positive integer $\rho$ where $\mathcal {P}$ is called the
switching index set. We assume the switching instants
 $t_0=0,t_1,t_2,\dots$ of $\sigma$  satisfy $t_{k+1}-t_{k}\geq\tau>0$ for
some constant $\tau$ and for any $k\geq0$,  and $\tau$ is called the
dwell time. $I_n$ denotes the
identity matrix of dimension $n$ by $n$.

\section{Linear Output Regulation}\label{sec1.2}

In this section, we review the  linear output regulation problem for the
 class of linear time-invariant systems as follows:
 \begin{equation}\label{c1.e1.2.1}
\begin{split}
\dot{x}(t) & =   A x(t) + B u(t) + E v(t),~~x(0)=x_0,~~t \geq 0,\\
y_m (t) &= C_m x (t) + D_m u (t) + F_m v (t),\\
e(t) & =   {C} x(t) + D u(t) + F v(t), \\
\end{split}
\end{equation}
where $x\in \R^{n}$, $y_{m}\in \R^{p_{m}}$, $e\in \R^{p}$ and
$u\in \R^{m}$ are the state, measurement output, error output, and
input of the plant, and $v \in \R^q$ is the exogenous signal
generated  by an
exosystem\index{Exosystem} of the following form \EQ \dot{v} (t) & = &  S
v (t),~~v(0)=v_0,~~t\geq0, \label{exo} \EN
where $S$ is some constant matrix.

Typically, the tracking error $e$ is the difference between the system output $y$ and the reference input $r$, i.e., $e = y - r$  where $y  = {C} x + D u + F_1 v $ for some matrix $F_1$ and
$r = F_2 v$ for some matrix $F_2$. Thus, in (\ref{c1.e1.2.1}), we have $F = F_1 - F_2$.  The tracking error $e$  is assumed to be measurable, but it may not be the only
measurable variable available for feedback control. Using the
measurement output feedback control allows us to solve the output
regulation problem for some systems which cannot be solved by the error output feedback control.

For convenience, we put  the plant (\ref{c1.e1.2.1}) and the
exosystem (\ref{exo}) together into the following form
 \begin{equation}\label{c1.e1.comp}
\begin{split}
\dot{x} & =   A x + B u + E v, \\
\dot{v} & =   S v, \\
y_m &= C_m x  + D_m u  + F_m v, \\
e & =   {C} x + D u + F v,  \\
\end{split}\end{equation}
 and call
(\ref{c1.e1.comp}) a composite system with $\mbox{\rm    col} (x,
v)$ as the composite state.

In general, some components of the exogenous signal $v$, say, the reference inputs are  measurable and some other components of the exogenous signal $v$, say, the unknown external disturbances are not measurable. Denote the unmeasured and measured components of $v$ by $v_u \in \R^{q_u}$ and $v_m \in \R^{q_m}$, respectively, where $0 \leq q_u, q_m \leq q$ with
 $q_u + q_m = q$. Then, without loss of generality,  we can assume $v_u$ and $v_m$ are generated, respectively, by the following systems:
 \EQ \label{vum}
 \dot{v}_u = S_u v_u, ~~ \dot{v}_m = S_m v_m,
 \EN
 for some constant matrices $S_u$ and $S_m$. (\ref{vum}) is still in the form of (\ref{exo}) with $v = \col (v_u, v_m)$ and $S = \mbox{block diag} [S_u, S_m]$.

To emphasize that $v$ may contain both measurable and unmeasurable components, we can rewrite the plant (\ref{c1.e1.2.1}) as
 \begin{equation}\label{c1.e1.2.1x}
\begin{split}
\dot{x}(t) & =   A x(t) + B u(t) + E_u v_u(t) + E_m v_m (t) ,~~x(0)=x_0,~~t\geq0,\\
y_m (t) & =   {C}_m x(t) + D_m u(t) + F_{mu} v_u(t) + F_{mm} v_m(t),\\
e(t) & =   {C} x(t) + D u(t) + F_u v_u(t) + F_m v_m(t).\\
\end{split}\end{equation}
where $E = [E_{u} ~ E_{m}]$, $F_{m} = [F_{mu} ~ F_{mm} ] $, and $F = [F_{u} ~ F_{m}]$.

%It is assumed that A special case of the dynamic
%measurement output feedback control is the dynamic error output
%feedback control when $C_m = C,  D_m = D,  F_m = F$, i.e.,
%$y_m = e$. In many cases, the error output $e$ is not the only
%measurable variable available for feedback control. Using the
%measurement output feedback control allows us to solve the output
%regulation problem for some systems which cannot be solved by the
%error output feedback control.

We will consider the following so-called measurement output feedback control law:
 \begin{equation}\label{c1.c2m}
\begin{split}
u & =  K_z z + K_y y_m, \\
\dot{z} &= {\cal G}_1 z + {\cal G}_2 y_m, \\
\end{split}\end{equation}
 where
 $z \in \R^{n_z} $ with $n_z$ to be specified later, and $K_z \in \R^{m \times n_z}$,  $K_y \in \R^{m \times p_m}$, ${\cal
G}_1  \in \R^{n_z \times n_z}$, ${\cal G}_2 \in \R^{n_z \times
p_m} $ are constant matrices.

This control law contains the following four types of control laws as special cases.

\begin{description}
\item 1. Full information  when $y_m = \col (x, v)$ and $n_z=0$:
\EQ u & = & K_1 x + K_2 v, \label{c1} \EN where $K_1 \in \R^{m
\times n}$ and $K_2 \in \R^{m \times q}$ are constant matrices.

\item 2. (Strictly proper) Measurement output
feedback when $K_y = 0$:
 \begin{equation} \label{c1.c2a}
\begin{split}
u & =  K_z z,   \\
\dot{z} &= {\cal G}_1 z + {\cal G}_2 y_m. \\
\end{split}\end{equation}

\item 3.  Error output
feedback when $y_m = e$ and $K_y = 0$:
\begin{equation} \label{c1.c2}
\begin{split}
u & =  K_z z,   \\
\dot{z} &= {\cal G}_1 z + {\cal G}_2 e. \\
\end{split}\end{equation}

\item 4. Combined error feedback and feedforward  when $y_m = \col (e, v)$:
\begin{equation} \label{c1.c2b}
\begin{split}
u & =  K_z z  + K_{v} v,  \\
\dot{z} &= {\cal G}_1 z + {\cal G}_{e} e +  {\cal G}_{v} v,
\end{split}\end{equation}
where $K_y = [0_{m\times p},   K_v]$, and ${\cal G}_2 = [{\cal G}_{e},  {\cal G}_{v}]$.

\end{description}

Needless to say that the control law (\ref{c1.c2m}) contains cases other than the control laws (\ref{c1}) to
(\ref{c1.c2b}).

From the second equation of the plant (\ref{c1.e1.2.1}) and the first equation of the control law (\ref{c1.c2m}), the control input $u$ satisfies
\EQQ
u = K_z z + K_y (C_m x + D_m u  + F_m v),
\ENN
or
\EQ
(I_m - K_y D_m) u = K_y C_m x +  K_z z  + K_y F_m v.
\EN
Therefore, the control law (\ref{c1.c2m}) is well defined if and only if $I_m - K_y D_m$ is nonsingular. It can be easily verified that
 the four control laws (\ref{c1}) to
(\ref{c1.c2b}) all satisfy $K_y D_m = 0$. Thus, in what follows, we will assume $K_y D_m = 0$ though it suffices to assume $I_m - K_y D_m$ is nonsingular.
As a result, we have
\EQ
 u = K_y C_m x +  K_z z  + K_y F_m v.
\EN

Under the assumption that $K_y D_m = 0$, the control law
(\ref{c1.c2m}) can be put as follows:
\begin{equation} \label{c1.conv}
\begin{split}
u &= K_y C_m x +  K_z z  + K_y F_m v,  \\
\dot{z} &= ({\cal G}_2 (C_m + D_m K_y C_m)) x  + ({\cal G}_1 + {\cal G}_2 D_m K_z) z + {\cal G}_2 (F_m + D_m K_y F_m) v.
\end{split}\end{equation}

Thus, the closed-loop system composed of the plant
(\ref{c1.e1.2.1}) and the control law
 (\ref{c1.c2m}) can be put as follows:
\begin{equation} \label{c1.close}
\begin{split}
\dot{x}_c  &= A_c  x_c  + B_c v,  \\
e &= C_c  x_c  + D_c v,
\end{split}\end{equation}
where $x_c
=\mbox{\rm    col} (x, z)$, and \EQQ A_{c} &= & \left [
\begin{array}{cc}
A + B K_y C_m & BK_z  \\
{\cal G}_2 (C_m + D_m K_y C_m)  &  {\cal G}_1 + {\cal G}_2 D_m K_z
\end{array} \right ],~~
B_{c} =  \left [ \begin{array}{c}
E + B K_y F_m  \\
{\cal G}_2 (F_m + D_m K_y F_m)
\end{array} \right ], \nnum \\
C_{c} &=& [C+DK_yC_m ~~~ D K_z], ~~ D_{c} = F + DK_y F_m.  \ENN

In particular, for the full information control law (\ref{c1}),  we have $y_m = \col (x, v)$ and $n_z = 0$. Thus, $x_c = x$, $K_y C_m = K_1$,  $K_y F_m = K_2$.
As a result,  we have \EQQ
A_c &= & A + B K_1,~~B_c = E + B K_2, \nnum \\
C_c &= & {C} + D K_1,~~D_{c} = F + D K_2.  \ENN

We now describe the linear output regulation problem as follows.
\begin{Problem}  \label{problem0002}
Given the plant \eqref{c1.e1.2.1} and the exosystem \eqref{exo}, find the control law of the form (\ref{c1.c2m}) such that the closed-loop system has the following properties.
\begin{itemize}
  \item \textbf{Property 1}: The matrix $A_c$ is Hurwitz, i.e., all the
eigenvalues of $A_{c} $ have negative real parts;
  \item \textbf{Property 2}: For any $x_c(0)$ and $v(0)$,
$\lim\limits_{t\rightarrow\infty}e(t)=0.$
\end{itemize}
\end{Problem}

At the outset, we list some  standard assumptions needed for solving Problem \ref{problem0002}.

\basm \label{A1.1} $S $ has no eigenvalues with negative real
parts. \easm

\basm \label{A1.2} The pair $(A, B)$ is
stabilizable.\index{Stabilizable} \easm

\basm \label{A1.3} The pair $\left ( \left [C_{m}  ~~  F_{mu}
\right ], \left [
\begin{array}{cc}
A  & E_u \\
0  & S_u
\end{array} \right ] \right )
$ is detectable.\index{Detectable}
\easm

\basm \label{A1.x3}
The following linear equations
\begin{equation} \label{regulator}
\begin{split}
{X}S&={A}{X}+{B}{U}+{E},\\
0&={C}{X}+{D}{U}+{F},
\end{split}\end{equation}
admit a solution pair $({X},{U})$.
\easm

\bremark Assumption \ref{A1.2} is made so that Property 1, that is, the exponential stability of $A_{c}$, can be achieved by the state feedback control. Assumption \ref{A1.3} together with
Assumption \ref{A1.2} is to render the exponential stability of
$A_{c}$ by the measurement output feedback control. Assumption \ref{A1.1} is made only for convenience and
loses no generality. In fact, if Assumption \ref{A1.1} is violated, then, without loss of generality, we can assume $S = \mbox{block~diag} ~[S_1, S_2] $ where $S_1$ satisfies
Assumption \ref{A1.1}, and all the eigenvalues of $S_2$ have negative real parts. Thus, if a control law of the form (\ref{c1.c2m}) solves Problem 1 with the exosystem being given by $\dot{v}_1 = S_1 v_1$, then the same control law solves Problem 1 with the original exosystem  $\dot{v} = S v$.
This is because Property 1 is
guaranteed by Assumption \ref{A1.2} and / or Assumption \ref{A1.3},
 and, as long as the closed-loop system
has Property 1,  Property 2 will not be affected by exogenous signals
that exponentially decay to zero.
\eremark

\bremark \label{c1.rem5}
Equations (\ref{regulator}) are known as the regulator
equations \cite{Fr}. It will be shown in Theorem  \ref{c1.the102} and Remark \ref{remoi} that,  under Assumptions \ref{A1.1} to \ref{A1.3}, Problem \ref{problem0002} is solvable by a control law of the form (\ref{c1.c2m}) only if the regulator equations
are solvable. Moreover, if Problem \ref{problem0002} is solvable at all,  necessarily,
the trajectory of the closed-loop system starting from any initial condition is such that
\EQ \label{c1.lin1}
\lim_{t\rightarrow\infty}(x(t)-{X}v(t))=0~~\mbox{and}~~\lim_{t\rightarrow\infty}(u(t)-{U}v(t))=0.
\EN
Therefore, ${X}v$ and ${U}v$ are the {\it steady-state} state and the {\it steady-state} input of the
closed-loop system signal at which the tracking error $e$ is identically zero. Thus,  the steady
state behavior of the closed-loop system is completely
characterized by the solution of the regulator equations.
%For
%convenience, in the sequel,  $X v (t) $ and $U v (t)$ are called
%{\it zero-error constrained state} and {\it zero-error constrained
%control}.
\eremark

\bremark By Theorem 1.9 of \cite{huang2004}, for any matrices $E$ and $F$,  the regulator equations
(\ref{regulator}) are solvable if and only if  \EQ \mbox{\rm    rank}
\left [
\begin{array}{cc}
A - \lambda I & B \\
C & D
\end{array} \right ]
= n+p,  ~~\forall~~\lambda \in \sigma(S). \label{tzeros} \EN
Nevertheless, for a
particular pair of $(E, F)$, the regulator equations may still
have a solution even if  \eqref{tzeros} fails.
%This happens
%when \EQ \mbox{\rm    vec} \left ( \left [
%\begin{array}{c}
%E \\
%F
%\end{array}
%\right ] \right ) \in\mbox{\rm    Im} ({Q} ). \label{e1.3.10} \EN
%\label{c1.rem5a}
\eremark

\section{Solvability of the Linear Output Regulation Problem}\label{sec1.3}

In this section, we will study the solvability of Problem \ref{problem0002}. Let us first present the following lemma on the closed-loop system.

\blemma \label{c1.lem1} Suppose, under the control law (\ref{c1.c2m}),  the closed-loop system
(\ref{c1.close}) satisfies Property 1, i.e., $A_c$ is Hurwitz. Then,  the closed-loop system
(\ref{c1.close}) also satisfies  Property 2, that is, $\lim_{t\rightharpoonup\infty}e(t)=0$, if
there exists a matrix $X_{c}$ that satisfies the
following matrix equations:
\begin{equation} \label{c1.syl}
\begin{split}
X_{c} S &= A_{c} X_{c} + B_{c}, \\
0   &= C_{c} X_{c} +  D_{c}.
\end{split}\end{equation}
Moreover,  under the additional Assumption \ref{A1.1}, the closed-loop system
(\ref{c1.close}) also satisfies  Property 2 only if
there exists a unique matrix $X_{c}$ that satisfies (\ref{c1.syl}).

\elemma

\bremark
The proof is the same as that of Lemma 1.4 of \cite{huang2004}, and is omitted. Here we only note that, if $X_{c}$ satisfies (\ref{c1.syl}), then  the variable $
\bar{x} = x_c - X_{c} v$ satisfies \EQQ
\dot{\bar{x}} &=& A_{c} {\bar{x}}, \\
e &=& C_{c} \bar{x}. \ENN
Since $A_{c}$
is Hurwitz, $\lim_{t \rightarrow \infty} \bar{x} (t)
= 0$, and hence, $\lim_{t \rightarrow \infty} e (t)
= 0$.
Since the solvability of the first equation of \eqref{c1.syl}  is guaranteed as long as the eigenvalues of $A_c$ do not coincide with those of $S$. Thus, Assumption \ref{A1.1} is not necessary for the sufficient part of
Lemma \ref{c1.lem1}. It suffices to
require that the eigenvalues of $A_c$ do not coincide with those of $S$.
\label{c1.rem4} \eremark

\blemma\label{c1.lem3}
Under Assumption \ref{A1.1}, suppose there exists a
control law  of the form (\ref{c1.c2m}) such that Property 1 holds. Then, Property 2 also holds  if and only if there exist matrices $X$ and $U$ that
satisfy the regulator equations
\begin{equation}\label{e1.3.16}\begin{split}
X S &= A X + B U + E,\\
0   &= C X + D U + F.
\end{split}\end{equation}
\elemma
The proof is similar to that of Lemma 1.13 of \cite{huang2004} and is thus omitted.

Now let us first consider the full information case where the
control law is defined by two constant matrices $K_1$ and $K_2$. The two
matrices $K_1$ and $K_2$ will be called the feedback
gain and the  feedforward
gain, respectively.

 \btheorem \label{c1.the1}  Under Assumption \ref{A1.2}, let the
feedback gain $K_1$ be such that $(A + B K_1 )$ is Hurwitz. Then, Problem 1 is solvable by the full information
control law (\ref{c1}) if
Assumption \ref{A1.x3} holds and the feedforward gain $K_2$
is given by
\EQ \label{fw}
K_2 = U - K_1 X. \EN \etheorem

\begin{proof} Under Assumption \ref{A1.2}, there exists $K_1$ such
that $A_c=A + B K_1 $ is  Hurwitz. Thus, under the control law (\ref{c1}),  Property 1 is satisfied.
Under Assumption \ref{A1.x3},  let $\bar{x} = x - X v $ and $K_2$ be given by (\ref{fw}). Then we have
\begin{equation}\label{e1.3.3} \begin{split}
\dot{\bar{x}} &= (A+ B K_1 ) \bar{x}, \\
e  &= (C+ D K_1 ) \bar{x}.
\end{split}\end{equation}
Since  $(A+ B K_1 )$ is Hurwitz, $\bar{x} (t)$ and hence $e (t)$ will approach zero as $t$ tends to infinity.  Thus, Property 2 is also satisfied.
 \end{proof}

\bremark \label{remoi0}
By Lemma  \ref{c1.lem3}, Assumption \ref{A1.x3} is also necessary for the solvability of Problem 1 by the full information
control law (\ref{c1})  if Assumption  \ref{A1.1} also holds.
\eremark

We now turn to the construction of the measurement output feedback control law (\ref{c1.c2m}).  Since we have already known how to synthesize a
full information control law which takes the plant state $x$
and the exosystem state $v$ as its inputs, naturally, we seek to
synthesize a measurement output feedback control law by estimating
the state $x$ and the unmeasurable exogenous signal $v_u$.  To this end,  lump
the state $x$ and the unmeasured exogenous signals $v_u$ together to obtain the
following system:
\begin{equation}\label{e1.3.3} \begin{split}
\left [
\begin{array}{c}
\dot{x} \\
\dot{v}_u
\end{array} \right ]
&=\left [ \begin{array}{cc}
A  & E_u \\
0  & S_u
\end{array} \right ]
\left [ \begin{array}{c}
{x} \\
{v}_u
\end{array} \right ]
+ \left [ \begin{array}{c}
B \\
0
\end{array} \right ]u + \left [ \begin{array}{c}
E_m \\
0
\end{array} \right ]v_m,  \\
y_m &= [C_{m}  ~ F_{mu}] \left [ \begin{array}{c}
{x} \\
{v}_u
\end{array} \right ] + D_m u + F_{mm} v_m.
\end{split}\end{equation}
Employing the well known Luenberger observer\index{Luenburger observer} theory suggests
the following observer based control law:
\begin{equation} \label{e1.s4} \begin{split}
u &=  [K_1 ~~ K_{2u}] z + K_{2m} v_m,   \\
\dot{z} & = \left [ \begin{array}{cc}
A  & E_u \\
0  & S_u
\end{array} \right ] z
+ \left [ \begin{array}{c}
B \\
0
\end{array} \right ] u + \left [ \begin{array}{c}
E_m \\
0
\end{array} \right ]v_m
+ L (y_m - [C_{m}~~ F_{mu}] z  - D_m u - F_{mm} v_m),
\end{split}\end{equation}
where $K_{2u} \in \R^{m \times q_u}$, $K_{2m} \in \R^{m \times q_m}$,
and $L \in \R^{(n+q_u) \times p_m} $ is an observer gain matrix.

The control law  (\ref{e1.s4}) can be put in the following form
\begin{equation}\label{e1.s5xx}  \begin{split}
u &= K_z z + K_{2m} v_m, \\
\dot{z} &= {\cal G}_1 z + {\cal G}_{21} y_m + {\cal G}_{22} v_m,
\end{split}\end{equation}
where \EQQ
K_z &=& [K_1 ~~~ K_{2u}], \nnum \\
{\cal G}_1 & = & \left [ \begin{array}{cc}
A  & E_u \\
0  & S_u
\end{array} \right ] +
\left [ \begin{array}{c}
B \\
0
\end{array} \right ] K_z - L (\left[
                              \begin{array}{cc}
                                C_m&  F_{mu} \\
                              \end{array}
                            \right]
 + D_m K_z), \nnum \\~~
  {\cal G}_{21} &= & L, ~~  {\cal G}_{22} = \left [ \begin{array}{c}
B \\
0
\end{array} \right ] K_{2m}+ \left [ \begin{array}{c}
E_m \\
0
\end{array} \right ]  - L ( F_{mm} + D_m K_{2m}).
 \ENN

Since $v_m$ is measurable, there exists a matrix $C_v$ such that $v_m = C_v y_m$. Thus the control law (\ref{e1.s5xx}) can be further put into the standard form (\ref{c1.c2m}) with
$K_y = K_{2m} C_v$ and  ${\cal G}_{2} = {\cal G}_{21} + {\cal G}_{22} C_v$.

\btheorem \label{c1.the102}  Under Assumptions \ref{A1.2} and
\ref{A1.3}, Problem 1  is solvable by the
measurement output feedback control law (\ref{c1.c2m}) if  Assumption \ref{A1.x3} holds.
\etheorem

\begin{proof} First note
that, by Assumption \ref{A1.2}, there exists a state feedback gain
$K_1$ such that $(A + B K_1)$ is Hurwitz, and, by
Assumption \ref{A1.3}, there exist matrices $L_1$ and $L_2$ such
that \EQQ A_L = \left [
\begin{array}{cc}
A  & E_u \\
0  & S_u
\end{array} \right ] -
\left [ \begin{array}{c}
L_1 \\
L_2
\end{array} \right ] \left[
                              \begin{array}{cc}
                                C_{m}&  F_{mu} \\
                              \end{array}
                            \right] =
\left [ \begin{array}{cc}
A - L_1 C_{m}& E_u - L_1 F_{mu}   \\
- L_2 C_{m} & S_u - L_2 F_{mu}
\end{array}
\right ] \ENN is Hurwitz. Let $(X, U)$ satisfy
the regulator equations, $L = \mbox{col} (L_1, L_2)$, and $K_2 = U - K_1 X$, and partition $K_2$ as $K_2 =  [K_{2u}, K_{2m}]$.
Let $\bar{x} = (x - Xv)$, $\bar{u} = (u - Uv)$, and $z_e = \left [ \begin{array}{c}
x  \\
v_u
\end{array}
\right ] - z$. Then, it can be verified that
\EQQ
\bar{u} &=& [K_1, K_{2u}] \left [ \begin{array}{c}
x  \\
v_u
\end{array}
\right ] - K_z z_e  + K_{2m} v_m - Uv \\
 &=& - K_z z_e +
 K_1 {x} + K_2 v - (K_2 + K_1 X) v \\
 &= &- K_z z_e + K_1 \bar{x}.
\ENN
In terms of $\bar{x}$ and $z_e$, the closed-loop system is given by
\begin{equation} \label{dotxc}   \begin{split}
\dot{\bar{x}}  &= A \bar{x} + B \bar{u} = (A+B K_1) \bar{x} -  B K_z z_e, \\
\dot{z}_e &= A_L z_e.
\end{split}\end{equation}
Let $A_c$ be the closed-loop system matrix. Then
$\sigma(A_c)$ = $\sigma(A+B K_1)
\cup \sigma(A_L)$. Thus Property 1 is satisfied. To show $ \lim_{t \rightarrow \infty} e (t) = 0$, first note that
(\ref{dotxc}) implies that $\lim_{t \rightarrow \infty} \bar{x} (t) = 0$ and $\lim_{t \rightarrow \infty} z_e  (t) = 0$. Then note that   $e  = Cx + Du + Fv = C (x - Xv) + D (u - Uv) + (C X + DU + F)v = C (x - Xv) + D (u - Uv) = (C + D K_1) \bar{x} - D K_z z_e $.
\end{proof}

\bremark \label{remoi}
By Lemma  \ref{c1.lem3}, Assumption \ref{A1.x3} is also necessary for the solvability of Problem 1 by  a
measurement output feedback control law of the form (\ref{c1.c2m}) if Assumption  \ref{A1.1} also holds.
\eremark

Specializing (\ref{e1.s5xx})
 to the two special cases with $v = v_u$ and $v = v_m$, respectively,  gives the following two corollaries  of Theorem \ref{c1.the102}.

\begin{Corollary} \label{c1.the102c1}  Under Assumptions \ref{A1.2} to \ref{A1.x3} with $v_u = v$, Problem 1    is solvable by the
following observer based control law:
\begin{equation} \label{e1.s4c1} \begin{split}
u &=  [K_1 ~~ K_{2}] z,  \\
\dot{z} & =  \left [ \begin{array}{cc}
A  & E \\
0  & S
\end{array} \right ] z
+ \left [ \begin{array}{c}
B \\
0
\end{array} \right ] u
+ L (y_m - [C_{m}~~ F_{m}] z  - D_m u ),
\end{split}\end{equation}
where $L$
is an observer gain matrix of dimension $(n+q)$ by $p_m$.
\end{Corollary}

\begin{Corollary} \label{c1.the102c2}  Under Assumptions \ref{A1.2} to \ref{A1.x3} with $v_m = v$, Problem 1  is solvable by the
following observer based control law:
\begin{equation} \label{e1.s4c2}   \begin{split}
u &=  K_1 z + K_2 v,  \\
\dot{z} & =  A z + Bu + E v
+ L (y_m - C_{m} z - F_{m} v - D_m u ),
\end{split}\end{equation}
where $L$
is an observer gain matrix of dimension $n$ by $p_m$.
\end{Corollary}

\section{Linear multi-agent systems and distributed observer}\label{sec1.4}

In this section, we turn to the cooperative output regulation problem for a group of linear  systems as follows:
\begin{equation} \label{eq020101}
\begin{split}
  \dot{x}_i & =A_ix_i+B_iu_i+E_{i} v,\\
     y_{mi} & =C_{mi}x_i+D_{mi}u_i+ F_{mi} v,\\
    e_i & =C_ix_i+D_iu_i+F_{i} v,\quad i=1,\dots,N,
\end{split}
\end{equation}
where $x_i\in \R^{n_i}$, $y_{mi}\in \R^{p_{mi}}$, $e_i\in \R^{p_i}$ and
$u_i\in \R^{m_i}$ are the state, measurement output, error output, and
input of the $i$th subsystem, and $v \in \R^q$ is the exogenous signal
generated by a so-called exosystem
as follows
\begin{equation}\label{exom}
     \dot{v} =S v, ~~ y_{m0} = C_{0} v,
\end{equation}
where $y_{m0} \in \R^{p_0}$ is the output of the exosystem.

Like the special case with $N=1$, the exogenous signal $v$ may also contain
 both unmeasured components  $v_u \in \R^{q_u}$ and measured components $v_m \in \R^{q_m}$,  where $0 \leq q_u, q_m \leq q$ with
 $q_u + q_m = q$. Then, like in (\ref{c1.e1.2.1}),  we can assume $v_u$ and $v_m$ are generated by (\ref{vum}). Correspondingly, we assume $C_0 = [0_{p_0 \times q_u}, C_{m0}]$
 for some matrix $
 C_{m0} \in \R^{p_0 \times q_m}$.
As a result, the plant (\ref{eq020101}) can be further written as follows.
\begin{equation} \label{eq020101um}
\begin{split}
\dot{x}_i &=A_ix_i+B_iu_i+E_{iu} v_{u}+  E_{im} v_{m}, \\
     y_{mi} & =C_{mi}x_i+D_{mi}u_i+F_{miu} v_{u}+  F_{mim} v_{m}, \\
    e_i & =C_ix_i+D_iu_i+F_{iu} v_{u}+  F_{im} v_{m},\quad i=1,\dots,N ,\\
\end{split}
\end{equation}
where $E_i = [E_{iu} ~ E_{im}]$, $F_{mi} = [F_{miu} ~ F_{mim} ] $, and $F_i = [F_{iu} ~ F_{im}]$.

Various  assumptions
are as follows.

\begin{Assumption}\label{ass0201}
$S$ has no eigenvalues with negative real parts.
\end{Assumption}

\begin{Assumption}\label{ass0202}
For $i=1,\dots,N$, the pairs $(A_i,B_i)$ are stabilizable.
\end{Assumption}

\begin{Assumption}\label{ass0203}
For $i=1,\dots,N$, the pairs $\left ( \left [C_{mi}  ~~  F_{miu}
\right ], \left [
\begin{array}{cc}
A_i  & E_{iu} \\
0  & S_u
\end{array} \right ] \right )
$ are detectable.
\end{Assumption}
\begin{Assumption}\label{ass0204}
The linear matrix equations
\begin{eqnarray}\label{eq020x}
\left.
  \begin{array}{r}
X_iS=A_iX_i+B_iU_i+E_i\\
0=C_iX_i+D_iU_i+F_i
  \end{array}
\right. \quad i=1,\dots,N,
\end{eqnarray}
have solution pairs $(X_i,U_i)$.
\end{Assumption}

\bremark \label{remdec}
The system (\ref{eq020101}) is still in the form of (\ref{c1.e1.2.1}) with $x = \mbox{col} (x_1, \dots, x_N)$,  $u = \mbox{col} (u_1, \dots, u_N)$, $y_m = \mbox{col} (y_{m1}, \dots, y_{mN})$,
$e = \mbox{col} (e_1, \dots, e_N)$. Thus, if the state $v$ of the exosystem can be used by the control $u_i$ of each follower, then, by Theorem \ref{c1.the1}, under Assumptions
 \ref{ass0202} and \ref{ass0204},   the output regulation
problem of the system  (\ref{eq020101}) and the exosystem (\ref{exom}) can be solved by the following  \textbf{full information control law:}
\begin{equation}\label{deccontrol3011}
u_i = K_{1i}x_i+K_{2i}v,  \quad i=1,\dots,N,
 \end{equation}
 where $K_{1i}$, $i=1,\dots,N$,  are such that
$A_i+B_iK_{1i}$ are Hurwitz, and $K_{2i}= U_i - K_{1i} X_i$.

Under the additional Assumption  \ref{ass0203}, there exist $L_{i} \in \R^{(n_i+q_u) \times p_{mi}}$ such that
\EQQ \left [
\begin{array}{cc}
A_i  & E_{iu} \\
0  & S_u
\end{array} \right ] -
L_i \left[
                              \begin{array}{cc}
                                C_{mi}&  F_{miu} \\
                              \end{array}
                           \right ] \ENN are Hurwitz.
                           Partition $K_{2i}$ as $K_{2i} = [K_{2iu}, K_{2im}]$ with
$K_{2iu} \in \R^{m_i \times q_u}$, $K_{2im} \in \R^{m_i \times q_m}$.
Then, by Theorem \ref{c1.the102},  under Assumptions \ref{ass0202} to \ref{ass0204},
 the output regulation
problem of the system  (\ref{eq020101}) and the exosystem (\ref{exom}) can be solved by the following
 \textbf{measurement output feedback control law:}
\begin{equation}\label{deccontrol3021}
\begin{split}
    u_i&= [K_{1i} ~~ K_{2iu}] z_i + K_{2im} v_m, \quad i=1,\dots,N,\\
  \dot{z}_i & =  \left [ \begin{array}{cc}
A_i  & E_{iu} \\
0  & S_u
\end{array} \right ] z_i
+ \left [ \begin{array}{c}
B_i \\
0
\end{array} \right ] u_i + \left [ \begin{array}{c}
E_{im} \\
0
\end{array} \right ]v_m   \\
&+ L_i (y_{mi} - [C_{mi}~~ F_{miu}] z_i  - D_{mi} u_i - F_{mim} v_m ) . \\
 \end{split}
\end{equation}

% where $y_{m0} = C_0 v_{m}$ for some constant matrix such that $(C_0, S_m)$ is controllable.

%\begin{equation}\label{control3021x}
%\left\{\begin{split}
%    u_i&=K_{1i}\xi_i+K_{2i}\eta_i \quad i=1,\dots,N.\\
%       \dot{z}_i&=A_iz_i+B_iu_i+E_i\eta_i\\
%    &\qquad+L_i(C_{mi}z_i+D_{mi}u_i+F_{mi}\eta_i-y_{mi}) \\
%    \dot{\eta}_i&=S\eta_i+\mu\left(\sum_{j\in \mathcal {N}_i(t)}a_{ij}(t) C_0  (\eta_j-\eta_i)\right)\\
%\end{split}\right.
%\end{equation}

%In either control law
%
%which relies on the information of $y_{mi}$ and $y_0$. This control scheme is called purely decentralized control scheme and is impractical when $N$ is large.
%What makes our problem interesting is that we will consider the so-called distributed control law to achieve our objective.
\eremark

Nevertheless, in practice,  the communication among different subsystems of (\ref{eq020101})  is subject to some constraints due to, say, the physical distance among these subsystems.
Thus,  the exogenous signal  $v$ or the measurable exogenous signal $v_m$ may not be available for the control $u_i$ of all the followers.
Since, typically, $e_i = y_i - y_0$,  the tracking error $e_i$ may not
be available for the control $u_i$ of all the followers.
To describe the communication constraints among various subsystems, we view the system (\ref{eq020101}) and the system (\ref{exom}) as a
multi-agent system with (\ref{exom}) as the leader and the $N$ subsystems
of (\ref{eq020101}) as the followers, respectively. Let
$\bar{\mathcal{G}}_{\sigma(t)}=(\bar{\mathcal{V}},\bar{\mathcal{E}}_{\sigma(t)})$\footnote{See Appendix for
a summary of graph.} with  $\bar{\mathcal{V}}=\{0,1,\dots,N\}$ and
$\bar{\mathcal{E}}_{\sigma(t)}\subseteq \bar{\mathcal{V}}\times
\bar{\mathcal{V}}$ for all $t\geq0$ be a switching graph, where the node $0$ is associated
with the leader system \eqref{exom} and the node $i$, $i = 1,\dots,N$,
is associated with the $i$th subsystem of the system
\eqref{eq020101}. For $i=0,1,\dots,N$, $j=1,\dots,N$, $(i,j) \in
\bar{\mathcal{E}}_{\sigma(t)}$ if and only if  $u_j$ can use
$y_{mi}$ for control at time instant $t$. Let
$\bar{\mathcal{N}}_i(t)=\{j~|~(j,i)\in \bar{\mathcal{E}}_{\sigma(t)}\}$
denote the neighbor set of agent $i$ at time instant $t$.

The case where the network topology is static can be viewed as a
special case of switching network topology when the switching index
set contains only one element. We will use the simplified notation $\bar{\mathcal{G}}$ to denote a static graph.

We will consider the following class of control laws.
\begin{equation} \label{ctrlg}
\begin{split}
 u_i&=f_i(x_i, \xi_i), \quad i=1,\dots,N, \\
 \dot{\xi}_i&= g_i(\xi_i,  y_{mi}, \xi_j, y_{mj},  j \in \bar{\mathcal{N}}_i(t)), \\
\end{split}
\end{equation}
%
%\begin{subequations}\label{ctrlg}
%    \begin{align}
%      u_i&=f_i(x_i, \xi_i) \label{ctrlg.1}\\
%      \dot{\xi}_i&= g_i(\xi_i,  y_{mi}, \xi_j, y_{mj},  j \in \bar{\mathcal{N}}_i(t)), \label{ctrlg.2}
%      \end{align}
%\end{subequations}
where both $f_i$ and $g_i$ are linear in their argument, and $g_i$ is  time-varying if the graph $\bar{\mathcal{G}}_{\sigma (t)}$ is. It can be seen that, at each time $t \geq 0$, for any $i = 1,\dots, N$, $u_i$ can make use of
 $y_{m0}$ if and only if the leader is a neighbor of the subsystem $i$. Such a control law is called a \textbf{distributed control law}. If, for  $i = 1,\dots, N$, $f_i$ is independent of $x_i$, then the control law is called a \textbf{distributed measurement output feedback control law}. If, for  $i = 1,\dots, N$,
 $\bar{\mathcal{N}}_i(t) = \{0\}$, $\forall~t\geq 0~$, then the control law (\ref{ctrlg}) is called a \textbf{purely decentralized control law}. In particular,
(\ref{deccontrol3011}) and (\ref{deccontrol3021}) are called the \textbf{purely decentralized full information control law}, and the \textbf{purely decentralized measurement output feedback control law}.

We now describe our problem as follows.
\begin{Problem}\label{def0201}
  Given the systems (\ref{eq020101}), (\ref{exom}) and a switching graph
  $\bar{\mathcal{G}}_{\sigma(t)}$, find a distributed control
  law of the form (\ref{ctrlg})
  such that the closed-loop system has the following properties:
\begin{itemize}
  \item \textbf{Property 1}: The origin of the closed-loop system  with $v$ being set to zero is
asymptotically stable.
  \item \textbf{Property 2}: For any initial condition $x_i(0)$, $\xi_i(0)$, $i=1,\dots,N$, and
$v(0)$, the solution of the closed-loop system  is
such that
\begin{equation}
\lim_{t\rightarrow\infty}e_i(t)=0,\quad i=1,\dots,N.
\end{equation}
\end{itemize}
\end{Problem}

Clearly, the solvability of the above problem not only depends on the dynamics of the systems (\ref{eq020101}), (\ref{exom}), but also  the property of the graph
$\bar{\mathcal{G}}_{\sigma(t)}$.  A typical assumption on the graph
$\bar{\mathcal{G}}_{\sigma(t)}$ is as follows.
\begin{Assumption}\label{ass301}
There exists a subsequence $\{i_k\}$ of $\{i:i=0,1,\dots\}$ with
$t_{i_{k+1}}-t_{i_k}< \nu$ for some positive $\nu$ such that every
node $i=1,\dots,N$ is reachable from the node $0$ in the union graph
$\bigcup_{j=i_{k}}^{i_{k+1}-1}\bar{\mathcal{G}}_{\sigma(t_j)}$.
\end{Assumption}

\bremark
Assumption \ref{ass301} is similar to what was proposed in
\cite{jadbabaie2003,ni2010,sh2}, and will be called \textbf{jointly connected} condition in the sequel.
Since, under Assumption \ref{ass301}, the graph $\bar{\mathcal{G}}_{\sigma(t)}$ can be disconnected for all $t \geq 0$,  it is perhaps  the least stringent condition on the graph as opposed to some other conditions such as every time connected, or frequently connected. In particular, Assumption \ref{ass301} is satisfied if Assumption 2 of \cite{mydj} is. Thus the main result in \cite{mydj}
is essentially included in \cite{sh2} even though the approach in \cite{mydj} appears somehow different from that in \cite{sh2}.
\eremark

The static graph is a special case of the switching graph when $\rho = 1$. For this special case, Assumption \ref{ass301}  reduces to the following.

\begin{Assumption}\label{ass301static}
Every
node $i=1,\dots,N$ is reachable from the node $0$ in the static graph
$\bar{\mathcal{G}}$.
\end{Assumption}

\begin{Remark}\label{remark010601}
Let  ${\cal G}_{\sigma(t)}=({\cal V},{\cal E}_{\sigma(t)})$  denote
the subgraph of $\bar{\mathcal {G}}_{\sigma(t)}$ where ${\cal
V}=\{1,\dots,N\}$, and ${\cal E}_{\sigma(t)} \subseteq {\cal V} \times
{\cal V}$ is obtained from $\bar{{\cal E}}_{\sigma(t)}$ by removing
all edges between the node 0 and the nodes in ${\cal V}$. Let $\bar{\mathcal{A}}_{\sigma(t)}=[a_{ij} (t)]_{i,j = 0}^N$ denote the weighted adjacency
matrix  of $\bar{\mathcal{G}}_{\sigma(t)}$,  let $\mathcal {L}_{\sigma(t)}$ be  the Laplacian matrix of
${\mathcal {G}}_{\sigma(t)}$ and $\Delta_{\sigma(t)} = \mbox{diag} (a_{10} (t), \dots, a_{N0} (t))$.
Then, it is shown in Remark 14 of \cite{sh3} that,  under Assumption \ref{ass301}, the matrix $H_{\sigma(t)}=\mathcal {L}_{\sigma(t)} +\Delta_{\sigma(t)}$
has the property that all the eigenvalues of the matrix $\sum_{j=i_{k}}^{i_{k+1}-1}  H_{\sigma(j)}$ have positive real parts. Furthermore, if the graph ${\cal G}_{\sigma(t)}$ is undirected, then the matrix $\sum_{j=i_{k}}^{i_{k+1}-1}  H_{\sigma(j)}$ is positive definite. In particular, under Assumption \ref{ass301static}, the constant matrix $-H$ is Hurwitz.
\end{Remark}

\section{Some Stability Results}

As pointed out in Introduction, our approach is based on the employment of the distributed observer. To introduce the  distributed observer. Let us first consider the
 stability property for the following class of switched linear systems:
\begin{equation}\label{eq010801}
    \dot{x}(t)=\left(I_N\otimes A - \mu F_{\sigma(t)}\otimes
    (B K)\right)x(t), \quad \sigma(t)\in \mathcal {P},
\end{equation}
where $A\in \R^{n\times n}$, $B\in \R^{n\times m}$, and  $F_{\sigma(t)} \in \R^{n\times n}$ are given, and $\mu>0$ and $K\in \R^{m\times
n}$ are to be designed.

\begin{Assumption}\label{ass10}
 There exists a subsequence $\{i_k\}$ of $\{i:i=0,1,\dots\}$ with
$t_{i_{k+1}}-t_{i_k} \leq \nu$ for some positive $\nu$ such that
 all the eigenvalues of the matrix $\sum_{j=i_{k}}^{i_{k+1}-1}  F_{\sigma(j)}$ have positive real parts.
\end{Assumption}

The stability property of the system of the form (\ref{eq010801}) has been extensively studied in the literature. We summarize the main results in two lemmas corresponding to
the switching network and the static network, respectively, as follows.

\blemma \label{thedo1}  Under Assumption \ref{ass10}, suppose the pair $(A, B)$ is controllable. Then,

(i) If $A$ is marginally stable, i.e., there exists a unique positive definite matrix $P$ such that $PA+A^TP\leq0$,  and $F_{\sigma(t)}$ is symmetric,  then,  with
$\mu = 1$, and $K = B^T P
$, (\ref{eq010801})  is asymptotically stable;

 (ii) If $B = I_n$ and $A$ has no eigenvalues with positive real parts, then, with  any
$\mu >0$, and $K = I_n$, (\ref{eq010801}) is asymptotically stable.
\elemma

\bremark
The stability property of the system of the form (\ref{eq010801})  was first studied in \cite{sh3}.   Part (i) of Lemma  \ref{thedo1} was established in Theorem 1 of  \cite{sh3}.
Part (ii) of Lemma  \ref{thedo1} was established in Lemma 2 of \cite{sh2}.
As a corollary of Lemma 2 of  \cite{sh2}, under Assumption \ref{ass10}, for any $\mu >0$,  the following system
\EQ
\dot{x} = - \mu F_{\sigma (t)} x
\EN
is asymptotically stable. As a special case of this result, when  $\mathcal {P} = \{1\}$, the matrix  $F_{\sigma (t)}$ is constant \cite{HH}, and will be denoted by $F$.
For this special case, the result of Lemma  \ref{thedo1} can be  strengthened to the following.
\eremark

\blemma \label{thedo2}  Under Assumption \ref{ass10} with  $\mathcal {P} = \{1\}$, suppose the pair $(A, B)$ is stabilizable. Then,

(i) Let $P$ be the unique positive definite matrix
satisfying $PA+A^TP- P  B B^T P +I_n \leq 0$, and $\mu \geq
\delta^{-1}$ where  $\delta = \min \{Re (\lambda_i (F) \} $.
Then, with $K = B^T P$,  (\ref{eq010801})  is asymptotically stable;

(ii) If $B = I_n$, then, for any $A$,
(\ref{eq010801})  is asymptotically stable with $K= I_n$, and sufficiently large $\mu$.

\elemma

\bremark

The proof of Part (i) of Lemma \ref{thedo2} can be extracted from the proof of  Theorem 2 of \cite{shh}. In fact, under Assumption \ref{ass10} with  $\mathcal {P} = \{1\}$,
all the eigenvalues of $F$ have positive real part. Let $T$ be such that $TFT^{-1}=J$ is in the Jordan form of $F$. Denote the eigenvalues of $F$ by $\lambda_1,\dots,\lambda_N$. Then
$(I_N\otimes A)-\mu(F\otimes B K)=(T^{-1} \otimes I_n)((I_N\otimes A)-\mu(J\otimes BK))(T\otimes I_n)$. From the block triangular structure of $J$, we know that
the eigenvalues of $(I_N\otimes A)-\mu(F\otimes B K)$ coincide with those of $A-\mu \lambda_i BK$, $i=1,\dots,N$. Since the pair $(A, B)$ is controllable,  by Lemma 1 of \cite{tuna},  the algebraic Riccati
equation
\begin{equation}A^T {P}+{P}A-{P} B B^T{P}+I_n=0\label{eq050610x}\end{equation}  admits a unique
positive definite solution ${P}$. Moreover, for any  $\mu \geq \delta^{-1}$ where  $\delta = \min \{Re (\lambda_i (F) \} $,
the gain matrix  $K= B^T P$ is such that $A-\mu \lambda_i B K $ and hence $(I_N\otimes A)-\mu(F \otimes  B K)$ are  Hurwitz.

Part (ii)  of Lemma \ref{thedo2} was established in Theorem 1 of \cite{sh1}.
It is also a direct result of the fact that the  eigenvalues of the matrix $(I_N\otimes A)-\mu(F\otimes
I_n)$ are
$$\{\lambda_i(A)-\mu\lambda_j(F):~i=1,\dots,n,~j=1,\dots,N\},$$ where
$\lambda_i(A)$ and $\lambda_j(F)$ are the eigenvalues of $A$ and
$F$, respectively. Thus,  the matrix $(I_N\otimes A)-\mu(F\otimes
I_n)$ is Hurwitz for sufficiently large $\mu$, and is Hurwitz for any positive $\mu$ if $A$ does not have any eigenvalue with positive real part.
\eremark

\section{Solvability of the Cooperative Linear Output Regulation Problem}\label{sec1.5}

Given systems (\ref{eq020101}), (\ref{exom}) and the switching graph
  $\bar{\mathcal{G}}_{\sigma(t)}$ whose weighted adjacency matrix is denoted by
  $\bar{\mathcal {A}}_{\sigma (t)}=[a_{ij} (t)]_{i,j=0}^N$, we call
 the following {compensator}
\begin{equation}\label{do}
    \dot{\eta}_i = S_m \eta_i +\mu L_0 \left(\sum_{j\in \bar{\cal
N}_i(t)}a_{ij}(t) C_{m0} (\eta_j-\eta_i)\right), ~ i = 1, \dots,  N,
 \end{equation} where $\eta_0 = v_m$, $\mu>0$ and $L_0 \in \R^{q_m \times p_0}$ are two design parameters,   a
{distributed observer candidate} for $v_m$, and call it
a {distributed observer} for $v_m$ if, for any $v_m(0)$ and
$\eta_i (0)$, $i = 1, \dots,  N$,
\begin{equation}
    \lim_{t \rightarrow \infty} ({\eta}_i (t) - v_m (t)) = 0,~~ i = 1, \dots,  N.
 \end{equation}

Whether or not (\ref{do}) is a {distributed
observer} of  $v_m$  depends on both the pair $(C_{m0}, S_m)$ and the
property of the graph.

Let $\tilde{\eta}_i = ({\eta}_i  - v_m ) $, and $ \tilde{\eta} = \mbox{col} (\tilde{\eta}_1, \dots, \tilde{\eta}_N)$. Then,
the system  \eqref{do}  can be put in the following compact form
\begin{equation} \label{cpdo}
    \dot{\tilde{\eta}}=\left ( (I_N\otimes S_m)-\mu (H_{\sigma (t)} \otimes L_0 C_{m0})\right )\tilde{\eta}.
    \end{equation}
Thus, the system  \eqref{do} is a {distributed observer} of $v_m$ if and only if the system \eqref{cpdo} is asymptotically stable.

\bremark
Since $( (I_N\otimes S_m)-\mu (H_{\sigma (t)} \otimes L_0 C_{m0}))^T = (I_N\otimes S^T_m)-\mu (H^T_{\sigma (t)} \otimes C^T_{m0} L^T_0)$, system (\ref{cpdo}) is asymptotically stable if and only if
a system of the form (\ref{eq010801}) with $A= S_m^T$, $F_{\sigma (t)} = H^T_{\sigma (t)}$, $B = C^T_{m0}$, and $K = L^T_0$ is asymptotically stable. Moreover, by Remark \ref{remark010601}, under Assumption \ref{ass301},  all the eigenvalues of the matrix $\sum_{j=i_{k}}^{i_{k+1}-1}  H_{\sigma(j)}$ have positive real parts. Furthermore, if the graph ${\cal G}_{\sigma(t)}$ is undirected, then the matrix $\sum_{j=i_{k}}^{i_{k+1}-1}  H_{\sigma(j)}$ is positive definite.
\eremark

Corresponding to the two purely decentralized control laws  (\ref{deccontrol3011}) and (\ref{deccontrol3021}),  we can synthesize two types of distributed control laws as follows:

\textbf{1. Distributed full information control law:}
\begin{equation}\label{control3011}
\begin{split}
  u_i &= K_{1i}x_i+K_{2i}\eta_i, \quad i=1,\dots,N,\\
  \dot{\eta}_i  &= S\eta_i+\mu L_0 \left(\sum\limits_{j\in \bar{\mathcal {N}}_i(t)}a_{ij}(t) C_0 (\eta_j-\eta_i)\right),
\end{split}
\end{equation}
where $K_{1i}\in \R^{m_i\times n_i}$ are such that
$A_i+B_iK_{1i}$ are Hurwitz, $K_{2i} = U_i - K_{1i} X_i$, $\mu$ is some positive constant, and
$\eta_0 = v = v_m$.
% where $y_{m0} = C_0 v_{m}$ for some constant matrix such that $(C_0, S_m)$ is controllable.

\textbf{2. Distributed measurement output feedback control law:}
\begin{equation} \label{control3021x}
\begin{split}
u_i &=  [K_{1i} ~~ K_{2iu}] z_i + K_{2im} \eta_i,  \quad i=1,\dots,N,\\
\dot{z}_i & = \left [ \begin{array}{cc}
A_i  & E_{iu} \\
0  & S_u
\end{array} \right ] z_i
+ \left [ \begin{array}{c}
B_i \\
0
\end{array} \right ] u_i + \left [ \begin{array}{c}
E_{im} \\
0
\end{array} \right ] \eta_i \\
&+ L_i (y_{mi} - [C_{mi}~~ F_{miu}] z_i  - D_{mi} u_i - F_{mim} \eta_i ),  \\
\dot{\eta}_i&= S_m \eta_i+\mu L_0 \left(\sum_{j\in \bar{\mathcal {N}}_i(t)}a_{ij}(t) C_{m0}  (\eta_j-\eta_i)\right),
\end{split}
 \end{equation}
where $\eta_0 = v_m$.

\bremark
The control law (\ref{control3011}) contains the distributed state feedback control law in \cite{sh2} as a special case by  letting $L_0 = C_0 = I_q$, or what is the same,  $y_{m0} = v$,
and  the control law (\ref{control3021x}) contains the distributed measurement output feedback control law in \cite{sh2} as a special case by  letting $v_m = v$, and $L_0 = C_{m0} = I_q$.
\eremark

\bremark
Both of the control laws  (\ref{control3011}) and (\ref{control3021x}) are synthesized based on the certainty equivalence principle in the sense that
 they  are obtained from the purely decentralized control laws (\ref{deccontrol3011}) and (\ref{deccontrol3021})
by replacing $v$ in (\ref{deccontrol3011}) and $v_m$ in (\ref{deccontrol3021}) with $\eta_i$, respectively, where $\eta_i$ is generated by a distributed observer of the form (\ref{do}).

\eremark

In \cite{sh1} and \cite{sh2}, the solvability of the cooperative output regulation problem was established by
means of Lemma \ref{c1.lem1}, which incurred tedious matrix manipulation.   In what follows, we will further simplify the proof of the solvability of the problem by means of the following Lemmas.

\blemma \label{lem0}
Consider the linear
time-invariant system
\begin{align}\label{c2.zetasub}
     \dot{x} = A x + B u,~ t  \geq 0,
\end{align}
where $x \in \R^n$, $A \in \R^{n \times n}$ is Hurwitz, and $u \in \R^m$ is piecewise
continuous in $t$ and $\lim_{t \rightarrow \infty} u (t) = 0$. Then, for any initial condition $x
(0)$, $\lim_{t \rightarrow \infty} x (t) = 0$.
\elemma

\begin{proof} The conclusion follows directly from the fact that the system \eqref{c2.zetasub} is input-to-state stable with the input $u$ decays to the origin asymptotically (Example 2.14 of \cite{chbook}). A more elementary self-contained proof can be given as follows. For any $x (0)$, let
\begin{align}
    x(T)=e^{AT}x(0)+\int_{0}^T e^{A(T- \tau)}Bu(\tau)d \tau, ~T\geq 0.
\end{align}
and
\begin{align}\label{eq-chp02.0x17}
    x(t)=e^{A(t-T)}x(T)+\int_{T}^te^{A(t- \tau)}Bu(\tau)d \tau,~ t \geq T.
\end{align}
It suffices to show $\lim_{t \rightarrow \infty} x (t) = 0$.
Since $A$ is Hurwitz, we have $||e^{A(t-T)}||\leq k e^{-\lambda(t-T)}$ for some $k>0$ and $\lambda>0$. Thus, for any $T \geq 0$,  $\lim_{t \rightarrow \infty} ||e^{A (t-T)} x (T)|| = 0$.
We only need to show that, for sufficiently large $T$,  $\lim_{t \rightarrow \infty} || \int_{T}^te^{A(t- \tau)}Bu(\tau)d \tau || =0$. In fact,  for any $t\geq T$,
\begin{align}
\left |\left | \int_{T}^te^{A(t- \tau)}Bu(\tau)d \tau \right |\right|  &\leq  \int_{T}^tk e^{-\lambda(t- \tau)}||B||\,||u(\tau)||d \tau \nonumber\\
   &\leq  \frac{k||B||}{\lambda}\sup_{T\leq \tau \leq t}||u(\tau)|| (1-e^{-\lambda(t-T)})\nonumber\\
   &\leq  \frac{k||B||}{\lambda}\sup_{T\leq \tau \leq t}||u(\tau)||.\label{eq-chp02.0x18}
\end{align}
Since $\lim_{t \rightarrow \infty} u (t) = 0$, for any $\varepsilon>0$, there exists $T>0$, such that, for any $t\geq T$, $||u(t)||\leq \frac{\lambda}{k||B||}\varepsilon$.
Thus,  \eqref{eq-chp02.0x18} implies
\begin{align}
 \left |\left |  \int_{T}^te^{A(\tau-t_0)}Bu(\tau)d \tau \right |\right|    &\leq {\varepsilon}, ~~ t \geq T.\label{eq-chp02.0x19}
\end{align}
Thus, $\lim_{t \rightarrow \infty} x (t) = 0$.
\end{proof}

\blemma  \label{lem00}  Under Assumption \ref{A1.1},
suppose a control law of the form (\ref{c1.c2m}) solves the output regulation problem of the system (\ref{c1.e1.comp}). Let
$\delta_u: [0, \infty) \rightarrow \R^m$ and $\delta_z:  [0, \infty) \rightarrow \R^{n_z}$ be two piecewise continuous time functions such that $\lim_{t \rightarrow \infty} \delta_u (t) = 0$,
and $\lim_{t \rightarrow \infty} \delta_z (t) = 0$. Then the following control law
\begin{equation} \label{purt}
\begin{split}
u & =   K_z z + K_y y_m + \delta_u (t),  \\
\dot{z} &= {\cal G}_1 z + {\cal G}_2 y_m + \delta_z (t),
\end{split}
\end{equation}
is such that $\lim_{t \rightarrow \infty} e (t) = 0$.
\elemma

\begin{proof} Denote the closed-loop system composed of (\ref{c1.e1.comp}) and (\ref{c1.c2m})  by (\ref{c1.close}). Then, $A_c$ is Hurwitz.  By Lemma \ref{c1.lem1}, there exists a unique matrix $X_c$ that satisfies equation (\ref{c1.syl}).  Let $\bar{x}_c (t) = x_c(t)-{X_c }v(t)$. Then the closed-loop system composed of (\ref{c1.e1.comp}) and (\ref{purt}) satisfies
\EQQ
\dot{\bar{x}}_c &=& A_{c} {\bar{x}}_c + B_u \delta_u (t) +B_z \delta_z (t),  \\
e &=& C_{c} \bar{x}_c + D \delta_u (t),
\ENN
where $B_u = \col (B, {\cal G}_2 D_m)$ and $B_z = \col (0_{n \times n_z}, I_{n_z})$.

By Lemma \ref{lem0}, we have $\lim_{t \rightarrow \infty} \bar{x}_c (t) = 0$, and, thus,
$\lim_{t \rightarrow \infty} e (t) = 0$.
\end{proof}

From the proof of Lemma \ref{lem00}, we can immediately obtain the following result.

\begin{Corollary} \label{coro3}
Under Assumption \ref{A1.1},
suppose a control law of the form (\ref{c1.c2m}) solves Problem 1. Let ${\cal S} (t)$ be a piecewise continuous square matrix defined over $[0, \infty)$ such that $\dot{\eta} = {\cal S} (t) \eta$ is asymptotically stable, and $K_{u}$ and $K_{z}$ be any constant matrices.
Then, under the following control law
\begin{equation} \label{purt2}
\begin{split}
u & =   K_z z + K_y y_m + K_{u} \eta ,   \\
\dot{z} &= {\cal G}_1 z + {\cal G}_2 y_m + K_{z} \eta , \\
\dot{\eta} &= {\cal S} (t) \eta, \\
\end{split}
\end{equation}
the closed-loop system also satisfies the two properties in Problem 1.
\end{Corollary}

We now consider the solvability of Problem 2.
%We are now ready to state the cooperative output regulation problem
%under switching network  as follows:
%\begin{Definition}\label{def0201}
%Given the follower system (\ref{eq020101}), leader system
%\eqref{exo} and a dynamic network topology $\bar{\mathcal
%{G}}_{\sigma(t)}$, find the control
%law \eqref{control3011} or \eqref{control3021} such that \\
%1) The origin of the system  $\dot{x}_c=A_{c,\sigma(t)}x_c$ is
%exponentially  stable.\\
%2) For any initial condition $x_i(0)$, $\eta_i(0)$, $i=1,\dots,N$, and
%$v(0)$, the solution of the closed-loop system (\ref{eq020201}) is
%such that
%\begin{equation}
%\lim_{t\rightarrow\infty}e_i(t)=0,\quad i=1,\dots,N
%\end{equation}
%\end{Definition}

\begin{Lemma}\label{lem301}
Suppose the distributed observer \eqref{cpdo}  is asymptotically stable. Then,

(i) Under Assumptions \ref{ass0201}, \ref{ass0202}, \ref{ass0204},  Problem 2  is solved
 by the distributed full information control law
\eqref{control3011};

(ii) Under the additional Assumption \ref{ass0203}, Problem 2  is
solved by the distributed  measurement output feedback control law (\ref{control3021x}).
\end{Lemma}

\begin{proof}
Part (i)
Let $K_{1i}\in \R^{m_i\times n_i}$ be such that
$A_i+B_iK_{1i}$ are Hurwitz, and $K_{2i} = U_i - K_{1i} X_i$, $i = 1, \dots, N$. Then, by Remark \ref{remdec},  the purely decentralized full information control law (\ref{deccontrol3011}) solves
Problem 2. Since the control law
\eqref{control3011} can be put in the following form:
\begin{equation} \label{eq020301}
\begin{split}
  u_i &= K_{1i}x_i + K_{2i} v + K_{2i} \tilde{\eta}_i, ~ \quad i=1,\dots,N,\\
 \dot{\tilde{\eta}} &= \left ( (I_N\otimes S )-\mu (H_{\sigma (t)} \otimes L_0 C_{ 0})\right ) \tilde{\eta}, \\
\end{split}
\end{equation}
where $\lim_{t \rightarrow \infty} \tilde{\eta} (t) = 0$. By Corollary \ref{coro3}, the proof is complete.

Part (ii)  Under the additional Assumption \ref{ass0203}, there exist   $L_{i} \in \R^{(n_i+q_u) \times p_{mi}}$ such that
 \EQQ A_{Li}  = \left [
\begin{array}{cc}
A_i  & E_{iu} \\
0  & S_u
\end{array} \right ] -
L_i \left[
                              \begin{array}{cc}
                                C_{mi}&  F_{miu} \\
                              \end{array}
                           \right ], ~ \quad i=1,\dots,N, \ENN are Hurwitz.
By Remark \ref{remdec}, Problem 2  can be solved by a control law of the form (\ref{deccontrol3021}).  Now denote the control law (\ref{deccontrol3021}) by $u_i = k_i (z_i, v_m), \dot{z}_i = g_i (z_i, k_i (z_i, v_m), y_{mi}, v_m)$, $i = 1, \dots, N$, and the control law
(\ref{control3021x})  by
\begin{equation} \label{decclose}
\begin{split}
 u_i &=   k_i (z_i, \eta_i),  \quad i=1,\dots,N,\\
\dot{z}_i &=  g_i (z_i, k_i (z_i,  \eta_i), y_{mi}, \eta_i), \\
 \dot{\tilde{\eta}} &= \left ( (I_N\otimes S_m)-\mu (H_{\sigma (t)} \otimes L_0 C_{m0})\right )\tilde{\eta}.\\
 \end{split}
\end{equation}
Then it is ready to verify that
\EQ \label{keta}
k_i (z_i, \eta_i) = k_i (z_i, \tilde{\eta}_i + v_m) = k_i (z_i, v_m) + K_{2im} \tilde{\eta}_i,
\EN
and \EQ \label{kzi}
&&g_i (z_i, k_i (z_i,  \eta_i), y_{mi}, \eta_i) \nnum \\
&=& g_i (z_i, k_i (z_i, v_m)+ K_{2im} \tilde{\eta}_i, y_{mi}, \tilde{\eta}_i + v_m) \nnum \\
&=& g_i (z_i, k_i (z_i, v_m), y_{mi}, v_m) +  \Gamma_i \tilde{\eta}_i,
\EN
where
$$\Gamma_i = \left [ \begin{array}{c}
                                B_{i} \\
                                 0
                                 \end{array}
                           \right ] K_{2im} + \left [ \begin{array}{c}
                                E_{im} \\
                                 0
                                 \end{array}
                           \right ] - L_i (F_{mim}  + D_{mi}K_{2im}).
                           $$
Since $\lim_{t \rightarrow \infty} \tilde{\eta}_i (t) = 0$ for $i = 1, \dots, N$,  the proof follows from Corollary \ref{coro3}.
\end{proof}

\bremark
Lemma \ref{lem301} is the reminiscent of the well known separation principle for the design of the Luenberger observer based output feedback control law.
What is worth noting is that the closed-loop system is a time-varying system when the graph is a switching graph.
\eremark

Combining Lemma \ref{lem301} with Lemmas \ref{thedo1} and  \ref{thedo2}, respectively, leads to the following two theorems.

\btheorem \label{themas1} Suppose Assumption  \ref{ass301} holds. Then,  \\
(i) if the graph $\bar{\mathcal{G}}_{\sigma(t)}$ is undirected,  the pair $(C_0, S)$ is observable, and $S$ is marginally stable, then, under Assumptions  \ref{ass0202} and \ref{ass0204}, Problem 2  is
solved by the distributed full information control law \eqref{control3011} with $\mu = 1$ and $L_0 = P C^T_{0}$ where $P$ is the unique positive definite solution of the inequality $P S^T + S P \leq 0$, and,
under the additional Assumption \ref{ass0203},  Problem 2  is
solved by the distributed  measurement output feedback control law (\ref{control3021x}) with $\mu = 1$ and $L_0 = P C^T_{m0}$ where $P$ is
the unique positive definite solution of the inequality $P S_m^T + S_m P \leq 0$;

(ii)   if  $y_{m0} = v$,  and none of the eigenvalues of $S$ has positive real part, then, under Assumptions  \ref{ass0202} and \ref{ass0204},
Problem 2  is solved by the distributed full information control law (\ref{control3011}) for any
$\mu >0$ and $L_0 = I_q$; and

(iii)   if  $y_{m0} = v_m$,  and none of the eigenvalues of $S_m$ has positive real part, then, under Assumptions  \ref{ass0202} to \ref{ass0204},
Problem 2  is solved by the distributed dynamic measurement output feedback control law (\ref{control3021x}) for any
$\mu >0$ and $L_0 = I_{q_m}$.
\etheorem

\btheorem \label{themas2} Suppose Assumption  \ref{ass301static} holds. Let  $\delta = \min \{Re (\lambda_i (H) \} $. Then, \\
(i) if  the pair $(C_0, S)$ is detectable, then, under Assumptions  \ref{ass0201}, \ref{ass0202} and  \ref{ass0204}, Problem 2 is
solved by the distributed full information control law \eqref{control3011} with $\mu \geq
\delta^{-1}$  and $L_0= P C_0^T$ where $P$ is the unique positive definite solution of the inequality
$PS^T + S P - P C^T_0 C_0 P + I_q \leq 0$,
and, under the additional Assumption \ref{ass0203}, Problem 2 is
solved  by the distributed measurement output feedback control law (\ref{control3021x}) with $\mu \geq
\delta^{-1}$ and $L_0= P C^T_{m0}$ where $P$ is the unique positive definite solution of the inequality
$P S^T_m + S_m P - P C^T_{m0} C_{m0} P + I_{q_m} \leq 0$;

(ii) if  $y_{m0} = v$, then, under Assumptions  \ref{ass0201}, \ref{ass0202} and \ref{ass0204},  Problem 2 is
solved by the distributed full information  control law \eqref{control3011} for sufficiently large $\mu$, and $L_0 = I_q$; and

(iii) if $y_{m0} = v_m$,
then, under Assumptions  \ref{ass0201} to \ref{ass0204},  Problem 2 is
solved by the distributed measurement output feedback control law (\ref{control3021x})
for sufficiently large $\mu$ and $L_0 = I_{q_m}$.
\etheorem

\begin{Remark}
In Parts (ii) and (iii) of Theorem \ref{themas2}, if none of the eigenvalues of $S$ or $S_m$ has positive real part, then
$\mu$ can be any positive real number.
\end{Remark}

\bremark
For the  case where $v = v_m$,  the control law (\ref{control3021x}) reduces to the following special form
\begin{equation} \label{e1.s4m}
\begin{split}
    u_i&=K_{1i} z_i+K_{2i}\eta_i, \quad i=1,\dots,N,\\
       \dot{z}_i&=A_iz_i+B_iu_i+E_i\eta_i +L_i(y_{mi} - C_{mi}z_i - D_{mi}u_i - F_{mi}\eta_i), \\
    \dot{\eta}_i&=S\eta_i+\mu L_0 \left(\sum_{j\in \mathcal {\bar{N}}_i(t)}a_{ij}(t) C_0  (\eta_j-\eta_i)\right), \\
\end{split}
\end{equation}
where $L_i\in \R^{n_i\times p_{mi}}$ are such that $(A_i - L_i C_{mi})$ are Hurwitz.

On the other hand, for the case where $v = v_u$, there is no measurable leader's signal $v_m$ to estimate,  the control law (\ref{control3021x}) reduces to the following special form
\begin{equation} \label{control3021u}
\begin{split}
u_i &=  [K_{1i} ~~ K_{2i}] z_i, \quad i=1,\dots,N,\\
\dot{z}_i & = \left [ \begin{array}{cc}
A_i  & E_{i} \\
0  & S
\end{array} \right ] z_i
+ \left [ \begin{array}{c}
B_i \\
0
\end{array} \right ] u_i + L_i (y_{mi} - [C_{mi}~~ F_{mi}] z_i  - D_{mi} u_i ).  \\
\end{split}
 \end{equation}
For this special case, the distributed observer is not needed and the control law is a purely decentralized one.
\eremark

\section{Some Variants and Extensions}

In this section, we will make some remarks on some variants and extensions of the problem studied in this chapter.

\subsection{Multiple leaders and containment control}
The containment control problem involves multiple leaders and the asymptotic tracking of the output of followers to a convex hull of the state variables of the multiple leaders \cite{r14}.
The problem formulation in Section \ref{sec1.4} also includes the containment control problem as a special case by appropriately interpreting the leader system and the tracking error $e_i$.
In fact, suppose there are multiple leaders of the following form:
\EQ \label{ml}
\dot{v}_i = S_i v_i, \; i = 1, \dots, l,
\EN
where, for $i = 1, \dots, l$, $v_i \in \R^{q_0}$ for some positive integer $q_0$, and $l$ is some integer greater than $1$.   Let $\mbox{Co} = \{\sum_{i=1}^{l} \alpha_i v_i, \alpha_i \geq 0, \sum_{i=1}^{l} \alpha_i = 1\} $. Then $\mbox{Co}$ is called the convex hull of the points $v_1, \dots, v_l$.
Let $v = \col (v_1, \dots, v_l)$

Now define, for $i = 1, \dots, N$, the output of the subsystem $i$ as $y_i = C_i x_i + D_i u_i + F_{1i} v$ for some matrix $F_{1i}$.
Denote the reference input of each follower by $r = F_{2} v$  where $F_2 = (\alpha_1, \dots, \alpha_l) \otimes I_{q_0}  $.
Let $e_i = y_i - r$. Then $e_i$ is in the form given in
(\ref{eq020101}) with $F_i = F_{1i} - F_2$. . Finally, let $S = \mbox{block diag} [S_1, \dots, S_l]$, and $C_0 = F_2$. Then the multiple leader systems (\ref{ml}) can be put in the standard form  (\ref{exom}).

It can be seen that the objective of making the tracking error $e_i$ approach the origin asymptotically implies the asymptotic convergence of the output of all follower subsystems to the convex hull $\mbox{Co}$.

\subsection{Local exogenous signals versus global exogenous signals}
Another variant of the  systems (\ref{eq020101}) is given as follows
\begin{equation} \label{eq020101l}
\begin{split}
  \dot{x}_i & =A_ix_i+B_iu_i+E_{i} v_i,\\
     y_{mi} & =C_{mi}x_i+D_{mi}u_i+ F_{mi} v_i,\\
    e_i & =C_ix_i+D_iu_i+F_{i} v_i,\quad i=1,\dots,N,
\end{split}
\end{equation}
where
\EQ \label{mlN}
\dot{v}_i = S_i v_i, \;~ i = 1, \dots, N,
\EN
for some constant matrices $S_i$. This formulation is actually contained in (\ref{eq020101}) by defining $v = \col (v_1, \dots, v_N)$,  and $S = \mbox{block diag} [S_1, \dots, S_N]$ and
redefining the matrices $E_{i}$, $ F_{mi}$, $F_{i}$, $i = 1, \dots, N$.

\subsection{Synchronized
reference generator and the output synchronization}

Given maps $\xi_i: [0, \infty) \rightarrow \R^p$ for $i = 1, \dots, N$ and a map $\bar{\xi}: [0, \infty) \rightarrow \R^p$, the elements of the set $\{\xi_i (\cdot): i =1, \dots, N\}$
are said to synchronize to $\bar{\xi} (\cdot)$ if $ \lim_{t \rightarrow \infty} ({\xi}_i (t) - \bar{\xi} (t)) = 0$ for all $i= 1, \dots, N$, and are said to synchronize if they synchronize to some
$\bar{\xi} (\cdot)$ \cite{tuna}.

Consider the following dynamic {compensator}
\begin{equation}\label{rdo}
    \dot{\eta}_i = S \eta_i +\mu L_0 \left(\sum_{j\in {\cal
N}_i(t)}a_{ij}(t) C_{0} (\eta_j-\eta_i)\right), ~ i = 1, \dots,  N,
 \end{equation} where $S \in \R^{q \times q}$,  $C_{0} \in \R^{p_0 \times q}$ are some given constant matrices,  ${\cal N}_i (t)$ denote the neighbor set of the node $i$ in the graph
 ${\mathcal{G}}_{\sigma(t)}$, and $\mu >0$ and $L_0 \in \R^{q \times p_0}$ are to be designed.

The compensator (\ref{rdo})  can be  obtained from (\ref{do}) by replacing  $\bar{\cal
N}_i (t)$ by  ${\cal
N}_i(t)$.

We assume the graph ${\mathcal{G}}_{\sigma(t)}$ satisfies the following assumption.

 \begin{Assumption}\label{ass501}
There exists a subsequence $\{i_k\}$ of $\{i:i=0,1,\dots\}$ with
$t_{i_{k+1}}-t_{i_k} \leq \nu$ for some positive $\nu$ such that the union graph
$\bigcup_{j=i_{k}}^{i_{k+1}-1} {\mathcal{G}}_{\sigma(t_j)}$ is connected.
\end{Assumption}

Then we have the following result.

\btheorem \label{themas10} Under Assumption  \ref{ass501}, \\
(i) Suppose the graph ${\mathcal{G}}_{\sigma(t)}$ is undirected,  $S$ is marginally stable, and the pair $(C_0, S)$ is observable. Then, with $\mu = 1$ and $L_0 = P C_0^T$ where the matrix $P$
is the unique positive definite matrix such that $P S^T + S P \leq 0$,
for any initial condition $\eta_i (0), ~i = 1, \dots, N$, the solution of
(\ref{rdo}) is such that
\begin{equation}
 \lim_{t \rightarrow \infty}  \left (   \eta_i (t) - e^{St} \frac{\sum_{j = 1}^{N} \eta_j (0)}{N} \right)  = 0
\end{equation}
exponentially.

(ii) If $C_0 = I_q$ and none of the eigenvalues of $S$ has positive real part, then, with $L_0 = I_q$ and any $\mu >0$,  and any initial condition $\eta_i (0), ~i = 1, \dots, N$, there exists some $\bar{\eta}_0 \in \R^q$ determined by $\eta_j (0),~ j = 1, \dots, N$ such that
 \begin{equation}
 \lim_{t \rightarrow \infty}  \left (   \eta_i (t) - e^{St}  \bar{\eta}_0 \right)  = 0
\end{equation}
exponentially.

\etheorem

\bremark
The proof of Part (i) of  Theorem \ref{themas10}  can be extracted from the proof of  Theorem 2 of \cite{sh3}.
Part (ii) of  Theorem \ref{themas10} was studied in Lemma 1 of \cite{scardovi09} which is in turn based on the result in \cite{moreau2004}. \eremark

\begin{Remark}
If the graph ${\mathcal{G}}_{\sigma(t)}$ is static and connected, and the pair $(C_0, S)$ is detectable, then, Theorem \ref{themas10} can be strengthened as follows. Let  $\lambda_2 (\mathcal{L})$
denote the smallest positive eigenvalue of
the Laplacian matrix $\mathcal{L}$ of the graph ${\mathcal{G}}$  and $P$
be the unique positive definite matrix such that $S {P}+{P}S^T-{P} C_0^T C_0{P}+I_q \leq 0$.  Then, for $\mu \geq \max\{1,   \frac{1}{\lambda_2 (\mathcal{L}) } \}$ and $L_0 = P C_0^T$,
the solution of
(\ref{rdo}) is such that
\begin{equation} \label{cpdostatic}
 \lim_{t \rightarrow \infty}  \left (   \eta_i (t) -  e^{St} \sum_{j = 1}^{N} r_j \eta_j (0) \right)  = 0,
\end{equation}
where $r = \col (r_1, \dots, r_N) \in \R^N$ is the unit vector  such that $r^T \mathcal{L} = 0$.
This special case  of Theorem \ref{themas10} is the direct result of  Lemma 1 of \cite{tuna}.
\end{Remark}

\begin{Remark}
The special case  with $\mu = 1$,  $C_0=I_q$, and $L_0 = I_q$ of the dynamic
compensator (\ref{rdo})
was proposed in \cite{wieland2011}, and was called synchronized
reference generators. Its main difference from  the distributed observer
\eqref{do} is that it does not contain a feedforward term {$a_{i0}(v-\eta_i)$}.
If we define  a virtual leader $\dot{v} = S v$, and let $\tilde{\eta}_i = ({\eta}_i  - v) $ and $ \tilde{\eta} = \mbox{col} (\tilde{\eta}_1, \dots, \tilde{\eta}_N)$, then,
the system  \eqref{rdo}  can be put in the following compact form
\begin{equation} \label{dcpdo}
    \dot{\tilde{\eta}}=\left (I_N\otimes S-\mu (L_{\sigma (t)} \otimes L_0 C_{0})\right )\tilde{\eta}.
    \end{equation}
By Theorem  \ref{themas10}, the compensator is not a distributed observer of the virtual leader $\dot{v} = S v$ because, for $i = 1, \dots, N$, the convergence of $\eta_i$ to $v$ happens only if $\eta_i (0)$ and $v (0)$ satisfy some equality. As a result, the control law \eqref{control3011} with the distributed observer  replaced  by the
synchronized reference generator (\ref{rdo})  will not solve the output regulation problem of (\ref{eq020101}) with the exosystem being $\dot{v} = S v,~ y_{m0} = C_0 v$.
Nevertheless, by the same technique as used in the proof of Lemma \ref{lem301},  it is possible to show that this control law can still make the output $y_i$, $i = 1, \dots, N$, of (\ref{eq020101}) synchronize to a signal of the form $e^{St}  \bar{\eta}_0$ for some $\bar{\eta}_0 $ determined by
$\eta_i (0)$,  $i = 1, \dots, N$.
\end{Remark}

\subsection{Discrete distributed observer}

To introduce our problem, let $\mathbb{Z}^+$ denote
the set of nonnegative integers, and $\sigma_d:\mathbb{Z}^+\rightarrow\mathcal {P}$ where $\mathcal
{P}=\{1,2,\dots,\rho\}$ is a piecewise constant switching signal in the
sense that there exists a subsequence $t_i$ of $\mathbb{Z}^+$, called switching
instants, such that $\sigma_d (t) = p$ for some $p \mathcal
{P}$  for $ t_i \leq t <
t_{i+1}$ for any $t_i\geq0$ and all  $t \in \mathbb{Z}^+$.

Consider the discrete-time counterpart of the linear system
(\ref{eq010801}) of the following form
\begin{equation}\label{eq010801d}
    {x}(t+1)=\left(I_N\otimes A - \mu F_{\sigma_d (t)}\otimes
    (B K)\right)x(t),~ t = 0, 1, \dots, \infty,
\end{equation}
where $A\in \R^{n\times n}$, $B\in \R^{n\times m}$,  $F_{\sigma_d (t)} \in \R^{n\times n}$ is a piecewise switching matrix, and
$\mu>0$ and $K\in \R^{m\times
n}$ are to designed.

\begin{Assumption}\label{ass10d}

(i) The switching times satisfy $t_{i+1}-t_i\geq \tau $ for some
positive integer $\tau >1$ for all $t_i$.

(ii) The matrix $F_{\sigma_d(t)}$ is symmetric for all $t \in \mathbb{Z}^+$,  and there exists a subsequence $\{i_k\}$ of $\mathbb{Z}^+$ with
$t_{i_{k+1}}-t_{i_k} \leq \nu$ for some positive $\nu$ such that
 all the eigenvalues of the matrix $\sum_{j=i_{k}}^{i_{k+1}-1}  F_{\sigma_d (j)}$ have positive real parts.
\end{Assumption}

Let $\bar{A}$ be the real Jordan form of $A$, $P$ be the nonsingular matrix such that $A = P^{-1} \bar{A} P$,   and $\bar{B} = P B$.  Then,
we have the following result:

\blemma \label{thedo1d}  Under  Assumption \ref{ass10d},  suppose all the eigenvalues of $A$ are semi-simple with modulus 1, and the pair $(A, B)$ is controllable. Then, with $K = B^T P^T PA$,  the system (\ref{eq010801d}) is asymptotically stable for all $\mu$ satisfying $$0 <
\mu\leq\min_{ p \in \mathcal {P}}\left\{\frac{1}{||
{F}_{p}\otimes
(\bar{A}^T\bar{B}\bar{B}^T\bar{A})||}\right\}.$$
 \elemma

\bremark
Lemma \ref{thedo1d} is taken from Lemma 3.1 of  \cite{sh4}.  As pointed out in Remark 2.2 of \cite{sh4}, the assumption that all the eigenvalues of $A$ are semi-simple with modulus 1 can be relaxed to the assumption that $A$ is marginally stable, i.e., all the eigenvalues of $A$ are inside the unit circle, and those eigenvalues of $A$ with modulus 1 are semi-simple.
\eremark

Now consider a discrete-time linear multi-agent system of the following form:
\begin{equation} \label{deq020101}
\begin{split}
  {x}_i (t+1)  & =A_ix_i (t) +B_iu_i (t) +E_{i} v (t),\\
     y_{mi}(t) & =C_{mi}x_i (t) +D_{mi}u_i (t) + F_{mi} v (t),\\
    e_i (t) & =C_ix_i (t)+D_iu_i (t) +F_{i} v (t), \quad i=1,\dots,N, ~ t = 0, 1, \dots, \infty,
\end{split}
\end{equation}
where $x_i\in \R^{n_i}$, $y_{mi}\in \R^{p_{mi}}$, $e_i\in \R^{p_i}$ and
$u_i\in \R^{m_i}$ are the state, measurement output, error output, and
input of the $i$th subsystem, and $v \in \R^q$ is the exogenous signal
generated by a discrete-time exosystem
as follows:
\EQ \label{dexo}
v (t+1) = S v (t), ~y_{m0} (t) = C_{0} v (t), ~ t = 0, 1, \dots, \infty,
\EN
where  $S \in \R^{q \times q}$ is marginally stable.  Let $\bar{\mathcal{G}}_{\sigma_d (t)}$ be a switching graph associated with  (\ref{deq020101}) and
 (\ref{dexo}) whose weighted adjacency matrix is denoted by
  $\bar{\mathcal {A}}_{\sigma_d (t)}=[a_{ij} (t)]_{i,j=0}^N$.

Define the following dynamic compensator
\begin{equation}\label{ddo}
   {\eta}_i (t+1) = S \eta_i (t) +\mu L_0 \left(\sum_{j\in \bar{\cal
N}_i(t)}a_{ij}(t) C_{0} (\eta_j (t) -\eta_i (t))\right), ~  i = 1, \dots,  N,
 \end{equation} where  $\eta_0 = v$, and the scaler $\mu>0$ and the matrix $L_0 \in \R^{q \times p_0}$ are to be designed.
The system (\ref{ddo}) is called  a
{discrete distributed observer candidate} for $v$, and is called
a {discrete distributed observer} for $v$ if, for any $v (0)$ and
$\eta_i (0)$,
\begin{equation}
    \lim_{t \rightarrow \infty} ({\eta}_i (t) - v (t)) = 0,~~ i = 1, \dots,  N.
 \end{equation}
Let $\tilde{\eta}_i = ({\eta}_i  - v ) $, and $ \tilde{\eta} = \mbox{col} (\tilde{\eta}_1, \dots, \tilde{\eta}_N)$. Then,
the system  \eqref{ddo}  can be put in the following compact form
\begin{equation} \label{dcpdo}
    {\tilde{\eta}} (t+1) =\left ( (I_N\otimes S)-\mu (H_{\sigma_d (t)} \otimes L_0 C_{0})\right )\tilde{\eta} (t),~ t = 0, 1, \dots, \infty.
    \end{equation}
Since $\left ( (I_N\otimes S)-\mu (H_{\sigma_d (t)} \otimes L_0 C_{0})\right )^T = \left ( (I_N\otimes S^T)-\mu (H^T_{\sigma_d (t)} \otimes L^T_0 C^T_{0})\right )$,
it is not difficult to deduce the conditions on various matrices and the graph for guaranteeing the stability property of (\ref{dcpdo}) from Lemma \ref{thedo1d}. Consequently, the discrete counterparts of Theorems \ref{themas1} and \ref{themas2}  can be obtained.

%which is in the same form as (\ref{eq010801d}), and

\subsection{Distributed adaptive observer}%
%******************************************************************************************************************

A drawback of the distributed observer (\ref{do}) is that the matrix $S$ or $S_m$ is used by the controller of every follower.  A more realistic controller should only allow those followers who are the children of the leader to know the matrix $S$ or $S_m$. In \cite{cai-huang-2015}, assuming $y_{m0} = v$ and $S_m = S$, a distributed  adaptive observer was proposed as follows:
\begin{equation} \label{doxx}
    \begin{split}
      \dot{S}_i&=\mu_1\sum_{j=0}^Na_{ij}(t)(S_j-S_i),~ i = 1,\cdots, N, \\
      \dot{\eta}_i&= S_i\eta_i+\mu_2\sum_{j=0}^Na_{ij}(t)(\eta_j-\eta_i)
      \end{split}
\end{equation}
where $S_i\in R^{q\times q}$,
$\eta_i\in R^q$, $S_0=S$, $\eta_0=v$, $\mu_1,\mu_2>0$.

Moreover, the following
result was  established  in \cite{cai-huang-2015}

\begin{Lemma}\label{lem33}
Consider the system \eqref{doxx}.   Under Assumption \ref{ass301}, suppose all the eigenvalues of the matrix $S$ are semi-simple with zero-real parts.
Then, for any $\mu_1,\mu_2>0$ and for any initial condition
   $S_i(0)$, $\eta_i(0)$ and $v(0)$, for $i=1,\dots,N$, $S_i(t)$ and $\eta_i(t)$ exist and are bounded
   for all $t\geq 0$, and
  \begin{equation}\label{ccl}
    \lim_{t\rightarrow\infty}(S_i(t)-S)=0,\ \lim_{t\rightarrow\infty}(\eta_i(t)- v(t))=0.
  \end{equation}
  \end{Lemma}

\begin{Remark} Lemma  \ref{lem33} holds for any $S$ if the graph is static and connected.
\end{Remark}

  \section{Concluding Remarks}

In  this chapter,  we have presented a unified framework for handling the cooperative output regulation problem of multi-agent systems using the distributed observer approach.
The main result not only contains various versions of the cooperative output regulation problem for linear multi-agent systems in the literature as special cases, but also
present a more general distributed observer. We have also simplified the proof of the main result by more explicitly utilizing the separation principle and
the certainty equivalence principle. In summary, we conclude that, as long as a distributed observer exists,
the cooperative output regulation problem of multi-agent systems is solvable if and only if
the classical output regulation problem of each subsystem is solvable by the classical way as summarized in Section 3.

\begin{acknowledgement}
This work has been supported in part by the Research Grants
Council of the Hong Kong Special Administration Region under grant
No. 412813, and in part by National Natural Science Foundation of
China under grant No. 61174049.
%If you want to include acknowledgments of assistance and the like at the end of an individual chapter please use the \verb|acknowledgement| environment -- it will automatically render Springer's preferred layout.
\end{acknowledgement}

\section*{Appendix}
\addcontentsline{toc}{section}{Appendix}

\section*{Appendix: Graph}
 A digraph $\mathcal {G}=(\mathcal {V},\mathcal {E})$
consists of a finite set of nodes $\mathcal {V}=\{1,\dots,N\}$ and an
edge set $\mathcal {E} = \{(i,j), i, j \in \mathcal {V}, i\neq j
\}$.  An edge from node $i$ to node $j$ is denoted by $(i,j)$. The node $i$ is called the father of the node $j$ and
the node $j$  the child of the node $i$. The node $i$ is also called the neighbor of node $j$. If the digraph
$\mathcal{G}$ contains a sequence of edges of the form $\{({i_1},
{i_2}), ({i_2}, {i_3}), \dots, ({i_{k}}, {i_{k+1}})\}$, then the set
$\{({i_1}, {i_2}), ({i_2}, {i_3}), \dots, ({i_{k}}, {i_{k+1}}) \}$
is called a directed path of $\mathcal{G}$ from ${i_1}$ to ${i_{k+1}}$, and
node ${i_{k+1}}$ is said to be reachable from node ${i_{1}}$.   A
digraph is said to be connected if it has a node from which there exists a directed path to every other node.  The edge $(i,j)$ is called
undirected if $(i,j)\in \mathcal {E}$ implies $(j,i)\in \mathcal
{E}$. The digraph is called undirected if every edge in $\mathcal {E}$
is undirected.
 A graph
$\mathcal{G}_s=(\mathcal{V}_s,\mathcal{E}_s)$ is called a subgraph
of $\mathcal{G}=(\mathcal{V},\mathcal{E})$ if
$\mathcal{V}_s\subseteq \mathcal {V}$ and $\mathcal{E}_s\subseteq
\mathcal{E}\cap (\mathcal{V}_s\times \mathcal{V}_s)$. Given a set of
$r$ graphs $\mathcal{G}_i=(\mathcal{V},\mathcal{E}_i)$, $i=1,\dots,r$,
the graph $\mathcal{G}=(\mathcal{V},\mathcal{E})$ with
$\mathcal{E}=\bigcup_{i=1}^r\mathcal{E}_i$ is called the union of
graphs $\mathcal{G}_i$ and is denoted by
$\mathcal{G}=\bigcup_{i=1}^r\mathcal{G}_i$. The weighted adjacency
matrix $\mathcal{A}=[a_{ij}]_{i,j = 1}^{N}$ of $\mathcal{G}$ is
defined as $a_{ii}=0$; for $i\neq j$,
 $a_{ij} >0  \Leftrightarrow (j,i)\in \mathcal{E}$, and $a_{ij} =  a_{ji}$ if the edge $(j,i)$ is undirected.
 The Laplacian of $\mathcal{G}$ is defined
 as $\mathcal{L}=[l_{ij}]_{i,j = 1}^{N}$, where $l_{ii} =\sum_{j=1}^Na_{ij} $, $l_{ij} =-a_{ij}$ for $i\neq j$. To define a switching graph, let $\mathcal{P}=\{1,2,\dots,\rho\}$ for
some positive integer $\rho$. We call a time function $\sigma:
[0,+\infty)\rightarrow \mathcal{P}=\{1,2,\dots,\rho\}$ a piecewise
constant switching signal if there exists a sequence $t_0=0 <t_1
<t_2,\dots$ satisfying $\lim_{i \rightarrow \infty} t_i = \infty$ such that, for any $k\geq 0$, for all $t
\in [t_k, t_{k+1})$, $\sigma (t) = i$ for some  $i \in \mathcal{P}$.
$\mathcal{P}$ is called the switching index set.
Given a piecewise
constant switching signal $\sigma:
[0,+\infty)\rightarrow \mathcal{P}=\{1,2,\dots,\rho\}$, and a set of
$\rho$ graphs $\mathcal{G}_i=(\mathcal{V},\mathcal{E}_i)$, $i=1,\dots,\rho$  with the corresponding weighted adjacency
matrices being denoted by $\mathcal{A}_i$, $i = 1,\dots, \rho$,  we call a time-varying graph $\mathcal{G}_{\sigma
(t)}=(\mathcal{V},\mathcal{E}_{\sigma (t)})$ a switching graph with the weighted adjacency
matrix $\mathcal{A}_{\sigma
(t)}$ if,  for any $k\geq 0$, for all $t
\in [t_k, t_{k+1})$, $\mathcal{A}_{\sigma
(t)}  = \mathcal{A}_i$ for some  $i \in \mathcal{P}$.

\end{document}